\documentclass[12pt]{article}
\usepackage{amsmath, amsfonts, amssymb, graphicx}%amsthm, amscd, color}
\usepackage[bookmarksnumbered, plainpages]{hyperref}

\usepackage[a4paper,top=2cm,bottom=2cm,left=2cm,right=2cm]{geometry}

\newtheorem{theorem}{Theorem}

\newtheorem{lemma}[theorem]{Lemma}

\newenvironment{proof}[1][Proof]{\noindent\textbf{#1.} }{\ \rule{0.5em}{0.5em}}

\begin{document}

\title{Groups with the real chain condition on\\non-pronormal subgroups}
\author{Ulderico Dardano\ \ {\small \textsl{and}}\ \ Fausto De Mari}
\date{}
\maketitle

\begin{center}
\vskip -0.95cm \texttt{{\small \phantom{ ......}dardano@unina.it%
\phantom{
.....}\ fausto.demari@unina.it}}\\[0pt]
\vskip 0.3cm Universit\`a degli Studi di Napoli ``Federico II'' \\[0pt]
\end{center}

\abstract{\noindent It is shown that a gerenalised radical group has no chain of  non-pronormal subgroups with the same order type as the set $\mathbb{R}$ of the real numbers if and only if either the group is minimax or all subgroups are pronormal.}

\bigskip \noindent \textbf{Mathematics Subject Classification (2020):} 20F22, 20E15, 20F19.

\smallskip \noindent \textbf{Keywords:} pronormal subgroup, minimax group, weak chain condition, deviation.

\vskip5mm

\section{Introduction}
The study of generalised soluble groups with a finiteness condition on the set
of subgroups not verifying a given property is a
standard problem in the theory of groups (see \cite{DS}). The properties which have
received the most attention are probably chain conditions on one hand and
generalization of normality on the other. 

Until now, the weakest  chain condition considered is the so-called {\it deviation} (see \cite{GKR,T03}). We do not state here the recursive definition of the deviation since, for our purposes, it is
enough to recall the fact that a poset has deviation if and only if it
contains no subposet order-isomorphic to the poset $D$ of all dyadic rationals $m/2^n$ in the interval  $0$ to $1$ (see \cite{CR}, 6.1.3). Since $D$ is a countable dense poset without endpoints, it is order-isomorphic to the rational numbers, by Cantor's isomorphism theorem. 
Therefore  \emph{a poset has deviation if and only if it
contains no subposet order-isomorphic to the poset $\mathbb{Q}$ of rational
numbers} with their usual ordering. Note that this condition generalises the so-called weak (double) chain conditions introduced by Zaicev \cite{Z} (see \cite{DDM} for more details).

Very recently, it has been shown  \cite{GKR} that \emph{if $G$ is a
radical group and the set of non-pronormal subgroups has deviation, then
either $G$ is minimax or all subgroups are pronormal}. Recall that a {\it radical group} is a group with an ascending (normal) series with locally nilpotent factors. Moreover, a  subgroup $X$ of a group $G$ is said to be {\it pronormal} if $X$ and $X^g$ are conjugate in $\langle X, X^g\rangle$ for every element $g$ of $G$; clearly, every normal subgroup is pronormal.  The structure of locally soluble groups in which all subgroups are pronormal has been long known \cite{KS} and groups with restrictions on the set of non-pronormal subgroups have been considered in many papers  \cite{GKR, dGV,dGV95,dGV00,V1,V2}.  Also, we must recall that a {\it minimax soluble-by-finite group} is a group with a finite series whose factors are cyclic, quasicyclic or finite. In fact, these are just soluble-by-finite groups with deviation on all subgroups \cite{T03}. 
In this paper we  extend the  result in \cite{GKR} to {\em gerenalised radical} groups, i.e. with an ascending (normal) series with locally nilpotent or locally finite factors, and show that we can consider a chain condition even  weaker  than deviation. 

A poset with the same order type as $\mathbb{R}$ will be called {\em $\mathbb{R}$-chain} and we will say that a poset has the  {\em real chain condition (RCC)} if it contains no {\em $\mathbb{R}$-chain} as subposet. Generalised radical groups with RCC on the poset of all subgroups with -- or without -- a given property have been recently considered in \cite{DDM}. Along the same lines we obtain here the  result below.

\medskip

\noindent
{\bf Theorem\;\;\;} {\it Let $G$ be a generalised radical group. Then $G$ satisfies the real chain condition on non-pronormal subgroups if and only if either $G$ is a soluble-by-finite minimax group or all subgroups of $G$ are pronormal.}

\medskip

Our notation and terminology are standard and can be found in \cite{R72}.

\section{ Proof of the Theorem}

 We first recall some basic fact concerning pronormal subgroups that
will be useful for our purposes. For more details about pronormal subgroups
of infinite groups we refer to \cite{dGV}.

\medskip

\begin{lemma}
\label{f1}\label{f2} Let $G$ be a group. Then:
\begin{itemize}
\item[(i)]If $H$ and $K$ are pronormal subgroups of $G$ such that $H^{K}=H$, then $HK$ is pronormal in $G$.
\item[(ii)] A subgroup of $G$ is normal if and only
if it is pronormal and ascendant. 
\item[(iii)] If $G$ is locally nilpotent, then every pronormal subgroup of $G$ is normal.
\end{itemize}
\end{lemma}

\begin{proof}
See \cite{dGV}, Corollary 2.8, Corollary 2.3 and Theorem 2.4.
\end{proof}

%%\medskip
%
%\begin{lemma}
%\label{f2}***IO METTEREI QUESTI DUE LEMMI IN UN UNICO ENUNCIATO*** Let $G$ be a group and let $H$ and $K$ pronormal subgroups of $G$.
%If $H^{K}=H$, then $HK$ is a pronormal subgroup of $G$.
%\end{lemma}
%
%\begin{proof}
%See \cite{dGV}, Corollary 2.8.
%\end{proof}
% 

\medskip
Let us now state a technical lemma which appears in \cite{DDM} in a more general setting. However, for the sake of completeness, we give a proof also here.

\medskip
\begin{lemma}
\label{1} Let $G$ be a group having a section $H/K$ which is the direct product of an infinite collection $(H_{\lambda }/K)_{\lambda \in \Lambda }$ of non-trivial subgroups, and let $L$ be a
subgroup of $G$ such that $L\cap H\leq K$ and $\left\langle H_{\lambda
},L\right\rangle =H_{\lambda }L$ for each $\lambda $.
If there is no $\mathbb{R}$-chain of non-pronormal subgroups of $G$ in the
interval $[H/K]$, then there exists a normal subgroup $H^{\ast }$ of $H$
containing $K$ such that $LH^{\ast }=H^{\ast }L$ is a pronormal subgroup of $G$.
\end{lemma}

\begin{proof}
Clearly the set $\Lambda$ may be assumed to be countable, so that it can be
replaced by the set $\mathbb{Q}$ of the rationals. Consider the subgroup $%
K_r=\underset{i<r}{\mathrm{Dr}}H_{i}$ for each $r\in\mathbb{R}$; then $%
\left\langle K_{r},L\right\rangle =K_{r}L$ for each $r\in\mathbb{R}$. Let $%
r_1,r_2\in\mathbb{R}$ with $r_1<r_2$, then $K_{r_1}< K_{r_2}$. If were $%
K_{r_1}L= K_{r_2}L$, since $L\cap K_{r_2}\leq L\cap H\leq K\leq K_{r_1}$,
Dedekind's Modular Law would give that 
\begin{equation*}
K_{r_2}=K_{r_2}L\cap K_{r_2}=K_{r_1}L\cap K_{r_2}=K_{r_1}(L\cap
K_{r_2})=K_{r_1};
\end{equation*}
this contradiction proves that $K_{r_1}L< K_{r_2}L$. Therefore $(K_rL)_{r\in\mathbb{R}}$ is an $\mathbb{R}$-chain and so $K_rL$ must
be a pronormal subgroup of $G$ for some $r\in\mathbb{R}$; hence the lemma holds with $H^{\ast }=K_r$.
%Clearly the set $\Lambda$ may be assumed to be countable, so that it can be
%replaced by the set $\mathbb{Q}$ of the rationals. Then the subgroup defined by $
%K_r=\underset{i<r}{\mathrm{Dr}}H_{i}$ for each $r\in\mathbb{R}$ form an $\mathbb{R}$-chain and the same do $%
%\left\langle K_{r},L\right\rangle =K_{r}L$ for each $r\in\mathbb{R}$. In fact, if 
%$
%r_1,r_2\in\mathbb{R}$ with $r_1<r_2$, then $K_{r_1}< K_{r_2}$. If were $%
%K_{r_1}L= K_{r_2}L$, since $L\cap K_{r_2}\leq L\cap H\leq K\leq K_{r_1}$,
%Dedekind's Modular Law would give that 
%\begin{equation*}
%K_{r_2}=K_{r_2}L\cap K_{r_2}=K_{r_1}L\cap K_{r_2}=K_{r_1}(L\cap
%K_{r_2})=K_{r_1};
%\end{equation*}
%this contradiction proves that $K_{r_1}L< K_{r_2}L$. Therefore $K_rL$ must
%be a pronormal subgroup of $G$ for some $r\in\mathbb{R}$ and the lemma holds with $H^{\ast }=K_r$.
\end{proof}

\medskip

\begin{lemma}
\label{p1} Let $G$ be a group with RCC on non-pronormal subgroups and let
$H/K$ be any section of $G$ which is the direct product of an infinite
family $(H_{\lambda}/K)_{\lambda \in \Lambda}$ of non-trivial subgroups.
Then $H$ is a pronormal subgroup of $G$; moreover, if $H$ is an ascendant
subgroup of $G$, then $H_{\lambda}$ is normal in $G$ for each $\lambda\in
\Lambda$.
\end{lemma}

\begin{proof}
Let $\lambda$ be any element of $%
\Lambda$. Clearly we may write $H/K=(H_{1}/K) (H_{2}/K)$ where both $H_{1}/K$ and $%
H_{2}/K$ are the direct product of an infinite collection of
non-trivial subgroups and $(H_{1}/K)\cap (H_{2}/K)=H_\lambda/K$. Application of Lemma \ref{1} yields that there exist an $H_{1}$-invariant subgroup $H_{1}^{\ast }$ in $[H_{1}/K]$ and an $H_{2}$-invariant subgroup $%
H_{2}^{\ast }$ in $[H_{2}/K]$ such that both $H_{1}^{\ast }H_{2}$ and $%
H_{1}H_{2}^{\ast }$ are pronormal subgroups of $G$. Clearly $H_{1}^{\ast
}H_{2} $ and $H_{1}H_{2}^{\ast }$ are both normal subgroups of $H$, so that $%
(H_{1}H_{2}^{\ast })^{(H_{1}^{\ast }H_{2})}=H_{1}H_{2}^{\ast }$ and hence $%
H=\left\langle H_{1}H_{2}^{\ast },H_{1}^{\ast }H_{2}\right\rangle $ is a pronormal subgroup of $G$ by Lemma \ref{f2}.

Since $H_{1}/K$ and $H_{2}/K$ are direct product of infinitely many non-trivial subgroups, arguing as in the first part of this proof we have that $H_{1}$ and $H_{2}$ are pronormal in $G$. Therefore if we assume that $H$ is ascendant, then $H_{1}$ and $H_{2}$ are likewise ascendant and so even normal by Lemma \ref{f1}, and so $H_{\lambda}=H_{1}\cap H_{2}$ is normal in $G$.\end{proof}

\medskip
In \cite{DDM} it is proved that {\it a gerenalised radical groups with RCC on non-normal subgroups  is either a soluble-by-finite minimax group or a Dedekind group}. We use this result in the next proof.

\medskip
\begin{lemma}
\label{za}Let $G$ be a group with RCC on non-pronormal subgroups, and let $H$ be any locally nilpotent non-minimax subgroup of $G$. Then $H$ is a Dedekind group; moreover, if $H$ is periodic, then $H$ is pronormal in $G$. In particular,
the Hirsch-Plotkin subgroup $R$ of $G$ is hypercentral and all subgroups of $%
R$ are ascendant in $G$.
\end{lemma}

\begin{proof}
Il follows from Lemma \ref{f1} that the locally nilpotent subgroup $H$ has RCC on non-normal subgroups so that $H$ is a Dedekind group by the above quoted result from\;\cite{DDM}.

Assume that $H$ is periodic, then $H=E\times A$ where $A$ is abelian and $E$
is either trivial or a quaternion group of order $8$ (see \cite{R96}, 5.3.7). Clearly $A$ is not minimax, so that it has a quotient $A/B$ which is the direct product of infinitely many non-trivial groups (see for instance \cite{GKR}, Lemma 3.2). Then $H/B$ is likewise a direct product of infinitely many non-trivial groups, and hence Lemma \ref{p1} yields that $H$ is pronormal in $G$.
%
%***
%
%Assume that $H$ is periodic, then $H=E\times A$ where $A$ is abelian and $E$
%is either trivial or a quaternion group of order $8$ (see for instance \cite%
%{R96}, 5.3.7). Write $A=D\times R****$ with $D$ divisible and $R$ reduced. If $%
%\pi(R)$ is infinite or $D$ is the direct product of infinitely many Pr\"ufer
%subgroups, then $H$ is the direct product of infinitely many non-trivial
%subgroups and so $H$ is pronormal in $G$ by Lemma \ref{p1}. Hence assume
%that $\pi(R)$ is finite and $D$ is the direct product of finitely many
%Pr\"ufer subgroups, therefore there is a prime $p$ such that the $p$%
%-component $R_p$ of $R$ is infinite. Then $R_p$ contains an basic subgroup $%
%B $ (see for instance \cite{R96}, 4.3.4), so that $B$ is a pure subgroup of $%
%R_p$ which is the direct product of cyclic subgroups and $R_p/B$ is
%divisible. Since any pure bounded subgroup of an abelian group is a direct
%factor (see for instance \cite{R96}, 4.3.8) and $R_p$ is reduced, the
%subgroup $B$ is infinite. On the other hand, $B$ is a homomorphic image of $%
%R_p$ (see \cite{S}, Theorem 5) and hence there exists a subgroup $S$ of $R_p$
%such that $R_p/S$ is the direct product of infinitely many non-trivial
%subgroups. Hence $H/S$ is likewise the direct product of infinitely many
%non-trivial subgroups and so $H$ is pronormal in $G$ by Lemma \ref{p1}.
%
%****
%
%

Finally, the first part of this proof gives that $R$ is either minimax or a Dedekind
group, in particular, since locally nilpotent groups of finite rank are
hypercentral (see \cite{R72} Part\;2, p.38), $R$ is hypercental and so all
subgroups of $R$ are ascendant in $G$.
\end{proof}

\medskip

\begin{lemma}
\label{p2}Let $G$ be a group with RCC on non-pronormal subgroups, and assume that the Hirsch-Plotkin radical $R$ of $G$ contains a subgroup which is the direct
product of infinitely many cyclic non-trivial subgroups. Then all subgroups
of $R$ are normal in $G$.
\end{lemma}

\begin{proof}
It follows from Lemma \ref{za} that 
%$R$ has no $\mathbb{R}$-chain of $G$-invariant subgroups, so that $R$ is a Dedekind group by Theorem $B$ (with $\chi$=normal); in
$R$ is a Dedekind group; in particular, all subgroups of $R$ are subnormal
in $G$. Let $X$ be a cyclic subgroup of $R$; the hypotesis imply that $R$
contains a subgroup $A$ which is the direct product of infinitely many
cyclic non-trivial subgroups such that $X\cap A=\{1\}$. Write $A=A_{1}\times
A_{2}$ where both $A_{1}$ and $A_{2}$ are the direct product of an infinite
collection of non-trivial cyclic subgroups. Application of Lemma \ref{p1}
yields that both $A_{1}X$ and $A_{2}X$ are normal in $G$. Therefore $%
X=A_{1}X\cap A_{2}X$ is likewise a normal subgroup of $G$. Thus all
subgroups of $R$ are normal in $G$.
\end{proof}

\medskip
Recall that the {\it total rank} of an abelian group is the sum of all $p$-ranks for $p=0$ or $p$ prime; in particular, an abelian group has finite total rank if and only if it is the direct sum of finitely many cyclic and quasicyclic groups and a torsion-free group of finite rank.

\medskip
\begin{lemma}
\label{asc}Let $G$ be a group with RCC on non-pronormal subgroups. If $A$ is an ascendant abelian subgroup of $G$ which is not a minimax group, then all
subgroups of $A$ are normal in $G$ and all cyclic subgroups of $G/A$ are
pronormal.
\end{lemma}

\begin{proof}
Clearly $A$ is contained in the Hirsch-Plotkin radical of $G$. In
order to prove that all subgroups of $A$ are normal in $G$, Lemma \ref{p2}
allows us to suppose that $A$ has finite total rank so that, in particular, since $A$ is not minimax, we have that $A$ is not periodic.  

Let $x$ be any element of infinite order of $A$.  Consider any free subgroup $E$ 
of $A$ such that $x\in A$ and $A/E$ is periodic; then $E$ is finitely
generated and $A/E$ has infinitely many primary components. Let $L$ be any
subgroup of finite index on $E$ containg $x$, then also $L$ is finitely
generated and $A/L$ has infinitely many primary components; let us denote by 
$A_p/L$ the $p$-component of $A/L$. By Lemma \ref{p1}, $A_p$ is $G$%
-invariant for each $p\in\pi(A/L)$. Thus 
\begin{equation*}
L=\underset{p\in\pi (A/L)}{\bigcap}A_{p}
\end{equation*}
is likewise normal in $G$. Since a well-know result by Mal'cev (see \cite{R96}, 5.4.16) yields that $%
\left\langle x\right\rangle $ is a closed subgroup of $E$, it follows that $%
\left\langle x\right\rangle $ is a normal subgroup of $G$. Hence all infinite
cyclic subgroups of $A$ are $G$-invariant, so that if $y\in A$ has finite
order also $\left\langle x,y\right\rangle =\left\langle x,xy\right\rangle $
is normal in $G$ and hence $\left\langle y\right\rangle $ is likewise normal
in $G$ since $\left\langle y\right\rangle $ is the torsion subgroup of $%
\left\langle x,y\right\rangle $. Therefore all subgroups of $A$ are normal
in $G$.

Consider now any cyclic subgroup $X$ of $G$ which is not contained in $A$, and let again $E$ be a free subgroup of $A$ such that $A/E$ is periodic; then $E$ is normal in $G$. 
Replacing 
$A/E$ with a suitable direct product of (infinitely many) its primary
components, we may clearly suppose that $(A/E)\cap (XE/E)$ is trivial and it is
possible apply Lemma \ref{1} to obtain that there exists a normal subgroup $%
A^{\ast}$ of $A$ containing $E$ such that $(XE)A^{\ast}=A^{\ast}(XE)$ is
pronormal in $G$. Thus $XA=((XE)A^{\ast})A$ is pronormal in $G$ by Lemma \ref%
{f2} and so $XA/A$ is likewise pronormal in $G/A$. Therefore all cyclic
subgroups of $G/A$ are pronormal.
\end{proof}

\medskip Recall that a group in which all subnormal subgroups are normal is
called a \textit{$T$-group}. The structure of soluble $T$-groups was
described by Gash\"utz and Robinson (see \cite{R64}). It turns out that
soluble $T$-groups are metabelian and that finitely generated soluble $T$%
-groups are either finite or abelian; moreover, every subgroup of a finite
soluble $T$-group is itself a $T$-group but this is no longer true for
infinite soluble $T$-groups. Groups in which all subgroups are $T$-groups
are called \textit{$\bar T$-groups}. It is easy to see that finite $\bar T$%
-groups are soluble so that, in particular, \textit{any locally
(radical-by-finite) $\bar T$-group is always metabelian and even abelian if
it is not periodic}.

The concept of pronormal subgroups arise naturally in the study of $\bar T$%
-groups. In fact, \textit{if all cyclic subgroups of a group $G$ are
pronormal then $G$ is a $\bar T$-group} (see \cite{dGV}, Lemma 3.2) and Peng 
\cite{P} proved that the converse also holds for finite soluble groups.
Infinite (gerenalised) soluble infinite groups in which all subgroups are
pronormal have been descrived in \cite{KS}, where an example of infinite
soluble periodic $\bar T$-group containing non-pronormal subgroups is also
given (thus, in particular, Peng's theorem does not hold for infinite
groups).

\medskip

\begin{lemma}
\label{lemma4} Let $G$ be a group with RCC on non-pronormal subgroups, and let $A$ be any
subgroup of $G$ which is the direct product of infinitely many cyclic
non-trivial subgroups. If $A$ is an ascendant subgroup of $G$, then $G$ is a 
$\bar{T}$-group.
\end{lemma}

\begin{proof}
Consider a subgroup $X$ of $G$ and a subnormal subgroup $Y$ of $X$. Suppose
first that $Y_0=Y\cap A$ is not finitely generated; hence also $Y_0$ is a
direct product of infinitely many non-trivial cyclic subgroups (see \cite{R96}, 4.3.16). 
Application of Lemma \ref{asc} yields that $Y_0$ is normal in $G$ and all
cyclic subgroups of $G/Y_0$ are pronormal. Therefore, as just quoted above, $%
G/Y_0$ is a $\bar{T}$-group and hence $Y$ is normal in $X$.

Assume now that $Y_0$ is finitely generated, thus $Y_0$ is contained in a
direct product of finitely many direct factors of $A$ and hence certainly
there exist two subgroups $A_1$ and $A_2$ which are both the direct product
of infinitely many direct factors of $A$ such that $A_1\cap A_2=\langle A_1,
A_2\rangle \cap Y=\{1\}$. Since all subgroups of $A$ are normal in $G$ by
Lemma\;\ref{asc}, application of Lemma \ref{1} gives that there exist two
subgroups $A_1^{\ast}\leq A_1$ and $A_2^{\ast}\leq A_2$ such that $%
A_1^{\ast}Y$ and $A_2^{\ast}Y$ are both pronormal in $G$. In particular, $%
A_1^{\ast}Y$ and $A_2^{\ast}Y$ are pronormal and subnormal in $AX$, so that
they are both normal in $AX$ by Lemma \ref{f1}. Thus $Y=A_1^{\ast}Y\cap
A_2^{\ast}Y$ is likewise normal in $AX$, and so also in $X$. Therefore $X$
is a $T$-group and so $G$ is a $\bar{T}$-group.
\end{proof}

\medskip

Now we deal with locally (radical-by-finite) groups.

\begin{lemma}
\label{asc'}Let $G$ be a locally (radical-by-finite) group with RCC on non-pronormal subgroups, and let $A$ be an ascendant abelian subgroup of $G$ which is
not a minimax group. If $G/A$ is not periodic, then $G$ is abelian.
\end{lemma}

\begin{proof}
By Lemma \ref{asc} all subgroups of $A$ are normal in $G$. If $B$ is a
subgroup of $A$ such that $A/B$ contains non-minimax subgroups $A_1/B$ and $%
A_2/B$ such that $A_1\cap A_2=B$, Lemma \ref{asc} gives that $G/A_1$ and $%
G/A_2$ are groups whose cyclic subgroups are pronormal, thus they are $\bar{T%
}$-groups %(see \cite{dGV}, Lemma 3.2) 
and hence they are abelian being locally (radical-by-finite) and
non-periodic; thus $G^{\prime}\leq A_1\cap A_2=B$. In particular, in order
to prove that $G$ is abelian it can be assumed that $A$ cannot be decomposed
as the direct product of two non-minimax subgroups; hence $A$ does not
contains subgroups which are the direct product of infinitely many
non-trivial cyclic subgroups and so $A$ has finite total rank. Let $E$ be a
free subgroup of $A$ such that $A/E$ is periodic. Since $E$ is residually
finite, it is enought to show that $G/L$ is abelian for each subgroup $L$ of
finite index of $E$. Since $A$ is not minimax but it has finite total rank, such a subgroup $L$ is finitely generated and the periodic abelian group $A/L$ has infinitely
many primary components. Hence, $A/L$ can be decomposed as the direct
product of two non-minimax subgroups and so $G^{\prime }\leq L$, as whished.
\end{proof}

\medskip For the sake of completeness let us prove the following elementary
property of abelian groups.

\begin{lemma}
\label{ab}Let $A$ be any torsion-free abelian group and assume that $%
A^{p^{n}}=A^{p^{n+1}} $ for some prime $p$ and $n\in\mathbb{N}$. Then there
exists a subgroup $B$ of $A$ such that $A/B$ is a group of type $p^{\infty}$.
\end{lemma}

\begin{proof}
Let $x_{0}$ be any element of $A^{p^{n}}$. Since $%
A^{p^{n}}=A^{p^{n+1}}=(A^{p^{n}})^{p}$, there exist elements $%
x_{1},x_{2},x_{3},...$ of $A^{p^{n}}$ such that $x_{0}=x_{1}^{p}$, $%
x_{1}=x_{2}^{p}$, $x_{2}=x_{3}^{p}, ...$. Put $P=\left\langle
x_{0},x_{1},x_{2},x_{3},...\right\rangle $, so that $P/\left\langle
x_{0}\right\rangle $ is a Pr\"{u}fer $p$-group. Then a well-know result of
Baer (see \cite{R96}, 4.1.3) yields that there exists a subgroup $%
B/\left\langle x_{0}\right\rangle $ such that $A/\left\langle
x_{0}\right\rangle =P/\left\langle x_{0}\right\rangle \times B/\left\langle
x_{0}\right\rangle $. Hence $A/B\simeq P/\left\langle x_{0}\right\rangle $
is a group of type $p^{\infty}$.
\end{proof}

\medskip

\begin{lemma}
\label{asc1}Let $G$ be a locally (radical-by-finite) group with RCC on non-pronormal subgroups, and let $A$ be an ascendant abelian subgroup of $G$ which is not a minimax group. If $A$ is torsion-free, then $A\leq Z(G)$.
\end{lemma}

\begin{proof}
By Lemma \ref{asc}, all subgroups of $A$ are normal in $G$ and Lemma \ref%
{asc'} allows us to suppose that $G/A$ is periodic. Assume, by
contradiction, that there exists $g\in G$ such that $[A,g]\neq\{1\}$. Since $%
g$ acts as a power automorphism on $A$, $g$ induces the inversion on $A$ (see \cite{R64}, Lemma 4.1.1) and 
$g^{2}\in C_{G}(A)$. Then, as $G/A$ is periodic, if $g^{n}\in A$, we have
that $g^{n}=g^{-n}$ and so $g$ is periodic. In particular, $%
A\cap\left\langle g\right\rangle =\{1\}$.

Note that we may assume that $A$ has finite rank. In fact, if not, $A$ would
contains two free subgroups of infinite rank $X$ and $Y$ such that $X\cap
Y=\{1\}$ and Lemma \ref{asc'} would imply that both factors $G/X$ and $G/Y$
are abelian, so that $G^{\prime}=\{1\}$ and the statement would be true.

Let $B$ any subgroup of $A$ such that $A/B$ is a group of type $2^{\infty}$.
Then $\left\langle g,A\right\rangle /\left\langle g^{2},B\right\rangle $ is
a locally dihedral $2$-group, hence it is hypercentral; in particular, it
follows that $\left\langle g,B\right\rangle $ is an ascendant subgroup of $%
\left\langle g,A\right\rangle $. On the other hand, $B$ is not minimax and
hence $\left\langle g,B\right\rangle /B$ is a pronormal subgroup of $G/B$ by
Lemma \ref{asc}. Therefore Lemma \ref{f1} yields that $\left\langle
g,B\right\rangle $ is even normal in $\left\langle g,A\right\rangle $, a
contradiction. It follows that $A$ has no quotients of type $2^{\infty}$,
hence application of Lemma \ref{ab} gives that $A^{2}\neq A^{4}$. Therefore $%
A/A^{4}$ contains some non-trivial element of order $4$.

As $A$ has finite rank, also $A/A^{4}$ has finite rank; on the other hand,
any primary abelian group of finite rank satisfy the minimal condition, and
hence we have that $A/A^{4}$ si finite. Thus $A^{4}$ is not minimax. Then
follows from Lemma \ref{asc} that $\left\langle g,A^{4}\right\rangle $ is
pronormal il $G$. Since $\left\langle g,A\right\rangle /\left\langle
g^{2},A^{4}\right\rangle $ is an abelian-by-finite $2$-group of finite
exponent, it is nilpotent (see \cite{R72} Part 2, Lemma 6.34) and so $%
\left\langle g,A^{4}\right\rangle $ is pronormal and subnormal in $%
\left\langle g,A\right\rangle $. Thus $\left\langle g,A^{4}\right\rangle $
is normal in $\left\langle g,A\right\rangle $ by Lemma \ref{f1}, and 
\begin{equation*}
\left\langle g,A\right\rangle /\left\langle g^{2},A^{4}\right\rangle
=\left\langle g,A^{4}\right\rangle /\left\langle g^{2},A^{4}\right\rangle
\times\left\langle g^{2},A\right\rangle /\left\langle
g^{2},A^{4}\right\rangle
\end{equation*}
but this is not possible as $g$ induces the inversion on $\left\langle
g^{2},A\right\rangle /\left\langle g^{2},A^{4}\right\rangle $ and $%
\left\langle g^{2},A\right\rangle /\left\langle g^{2},A^{4}\right\rangle $
has exponent (exactly) $4$. This contradiction proves that $A\leq Z(G)$.
\end{proof}

\medskip

\begin{lemma}
\label{hm} Let $G$ be a locally (radical-by-finite) group with RCC on non-pronormal subgroups, and assume that the Hirsch-Plotkin radical $R$ of $G$ is not
minimax. Then $G$ is a $\bar{T}$-group.
\end{lemma}

\begin{proof}
Assume, by contradiction, that $G$ is not a $\bar{T}$-group. Lemma \ref{za}
yields that $R$ is a Dedekind group, and so the torsion subgroup $T$ of $R$
is a Chernikov group as a consequence of Lemma \ref{lemma4} appied to the socle of $R$. Hence $%
T\neq R$, so that $R$ is non periodic. Thus $R$ is abelian and so $R=T\times
A$ by a suitable torsion-free subgroup $A$ (see for instance \cite{R96},
4.3.9); in particular, $A$ is not minimax. By Lemma \ref{asc}, all subgroups
of $R$ are normal in $G$ so that $G$ acts on $R$ either trivially or as the
inversion map (see \cite{R64}, Lemma 4.1.1); on the other hand, $A$ is a
torsion-free (non-trivial) subgroup of $R$ which is contained in $Z(G)$ by
Lemma \ref{asc1} and hence $G$ must acts trivially on $R$. Thus $R\leq Z(G)$.

Clearly $G$ is not abelian, hence $G/A$ is periodic by Lemma \ref{asc'}. Let 
$E$ be a free subgroup of $A$ such that $A/E$ is periodic. Then Lemma %
\ref{lemma4} yields that $E$ is finitely generated, so that $A/E$ does not
satisfy the minimal condition. Since $E$ is residually finite, in order to
prove that $G$ is soluble it is enought to prove that $G^{\prime\prime}$ is
contained in each subgroup of finite index of $E$. Let $L$ be any subgroup
of finite index of $E$. Then $L$ is a central subgroup of $G$ and $A/L$ does
not satisfy the minimal condition, hence the socle of $A/L$ is the direct
product of infinitely many non-trivial cyclic subgroups and so $G/L$ is a $%
\bar{T}$-group by Lemma \ref{lemma4}; in particular, since $G/L$ is
locally (radical-by-finite), it follows that $G/L$ is metabelian. Thus $%
G^{\prime\prime}\leq L$ and so, as noted above, this is enought to guarantee
that $G$ soluble. Therefore $C_{G}(R)\leq R$ and so, as $R\leq Z(G)$, we
have that $G=R$ is abelian. This contradiction proves the lemma.
\end{proof}

\medskip As already noted, a group whose (cyclic) subgroups are pronormal is
a $\bar T$-group, moreover, any locally (radical-by-finite) $\bar T$-group
is soluble (actually metabelian). Notice that also locally
(radical-by-finite) minimax groups are soluble-by-finite. In fact, it is
known that any radical minimax group is soluble (see \cite{R72} Part 2,
Theorem 10.35), while locally (soluble-by-finite) minimax groups are
soluble-by-finite (see \cite{DES}, Lemma 8).

\medskip

\begin{lemma}
\label{LRF} Let $G$ be a locally (radical-by-finite) group with RCC on non-pronormal subgroups. Then either $G$ is minimax or all subgroups of $G$ are
pronormal. In particular, $G$ is soluble-by-finite.
\end{lemma}

\begin{proof}
Assume first that $G$ is a $\bar{T}$-group which is not minimax; in
particular, $G$ is metabelian. Since locally (radical-by-finite)
non-periodic $\bar T$-groups are abelian, in order to prove that all subgroups of 
$G$ are pronormal, we may further suppose that $G$ is periodic. Thus, if $%
L=[G^{\prime},G]$, we have $\pi(L)\cap\pi(G/L)=\emptyset$ and $%
2\not\in\pi(L) $ (see \cite{R64}, Theorem 6.1.1). Let $\pi=\pi(G/L)$ and
assume, by contradiction, that there exists a Sylow $\pi$-subgroup $P$ of $G$
such that $PL$ is a proper normal subgroup of $G$. Suppose that $P$ is not a
Chernikov group. Since $P\simeq PL/L$ is nilpotent, Lemma \ref{za} yields
that $P$ is pronormal in $G$ and hence $G=P^GN_G(P)=PLN_G(P)$. Since $N_G(P)/P$ is
clearly a $\pi^{\prime}$-group also 
\begin{equation*}
G/PL\simeq N_G(P)/(PL\cap N_G(P))
\end{equation*}
is a $\pi^{\prime}$-group, a contradiction which proves that $P$ is a
Chernikov group. Let $F/PL$ be a finite non-trivial subgroup of $G/PL$.
Since $G$ acts as a group of power automorphisms of $L$, the factor $%
G/C_G(L) $ is residually finite and so the largest divisible subgroup $D$ of 
$P$ is contained in $C_G(L)$. Therefore $F/C_F(L)$ is finite and so $F=XL$
for a suitable subgroup $X$ such that $X\cap L=\{1\}$ (see \cite{D}, Theorem
2.4.5). But $F/DL$ is finite, hence $F$ is abelian-by-finite and so its
Sylow $\pi$-subgroups $P$ and $X$ are conjugate, which is not possible by
the choice of $X$. Therefore $G=PL$ for each Sylow $\pi$-subgroup $P$ of $G$%
. It follows that all subgroups of $G$ are pronormal (see \cite{KS}) and
hence the lemma holds when $G$ is a $\bar T$-group.

Assume now that $G$ is not a $\bar T$-group and, by contradiction, suppose
that $G$ is not minimax. Hence Lemma \ref{hm} yields the Hirsch-Plotkin
radical of $G$ is minimax, so that in particular all ascendant abelian
subroups of $G$ are minimax. Let $N$ be any normal subgroup of $G$, and let $%
H/N$ be the Hirsch-Plotkin radical of $G/N$. Then Lemma \ref{hm} yields that
either $G/N$ is a $\bar{T}$-group (and hence metabelian) or $H/N$ is minimax
(and hence soluble), thus the last non-trivial term of the derived series of 
$H/N$ is a non-trivial abelian normal subgroup of $G/N$. It follows that $G$
is hyperabelian and hence $G$ is minimax (see \cite{R72} Part 2, p.175), a
contradiction which concludes the proof.
\end{proof}

\medskip
%We are now in position to prove the main theorem.
%
%\begin{theorem}
%\label{teopn} Let $G$ be a gerenalised radical group. Then the following are
%equivalent:
%
%\begin{itemize}
%\item[(i)] $G$ satisfies the weak minimal condition on non-pronormal
%subgroups;
%
%\item[(ii)] $G$ satisfies the weak maximal condition on non-pronormal
%subgroups;
%
%\item[(iii)] $G$ satisfies the weak double condition on non-pronormal
%subgroups.
%
%\item[(iv)]  $G$ satisfies the real chain condition on non-pronormal subgroups.
%
%\item[(v)] either $G$ is a soluble-by-finite minimax group or all subgroups of $G$ are pronormal.
%\end{itemize}
%\end{theorem}

\medskip
Now we are in a position to prove a Theorem that contains the one stated in the introduction. In the next statement, {\it pronormal deviation} means that the set of non-pronormal subgroups has deviation (see \cite{GKR}) and, as quoted in the introduction, this is equivalent to require that there are no subsets of non-pronormal subgroups order isomorphic to the set $\mathbb{Q}$ of rational numbers with their usual ordering. 
%
%
%Let $\mathfrak{X}$ be a class of groups which closed with
%respect to forming subgroups and homomorphic images, such that each locally
%finite $\mathfrak{X}$-group is soluble-by-finite. Then any gerenalised
%radical $\mathfrak{X}$-group $G$ is radical-by-finite,

\medskip
\begin{theorem} Let $G$ be a gerenalised radical group. Then the following are
equivalent:
\begin{itemize}
\item[(i)] $G$ satisfies the real chain condition on non-pronormal subgroups;

\item[(ii)] $G$ has pronormal deviation;

\item[(iii)] $G$ satisfies the weak minimal condition on non-pronormal
subgroups;

\item[(iv)] $G$ satisfies the weak maximal condition on non-pronormal
subgroups;

\item[(v)] $G$ satisfies the weak double condition on non-pronormal
subgroups.

\item[(vi)] either $G$ is a soluble-by-finite minimax group or all subgroups of $G$ are pronormal.
\end{itemize}
\end{theorem}

\begin{proof}
Let $G$ be a gerenalised radical group with the real chain condition on non-pronormal subgroups. Clearly, any section of $G$ has likewise the real chain condition and so Lemma \ref{LRF} yields that each locally (radical-by-finite) section of $G$ is soluble-by-finite; therefore, since $G$ is gerenalised radical, by transfinite induction can be obtained that $G$ itself is soluble-by-finite and thus again Lemma \ref{LRF} gives that either $G$ is minimax or all subgroups of $G$ are pronormal. Hence $G$ satisfies both the weak minimal and weak maximal condition on non-pronormal subgroups. 

On the other hand, the weak minimal condition on non-pronomal subgroups, as well as the weak maximal condition on non-pronormal subgroups, imply the weak double chain condition on non-pronormal subgroups and so also the pronormal deviation which in turn implies the real chain condition on non-pronormal  subgroups (see \cite{DDM}).\end{proof}

%%%%%%%%%%%%%%%%%%%%%%%%%%%%%%%%
%\bibitem{BdG} M. Brescia and F. de Giovanni, Groups satisfying the double chain condition on non-pronormal subgroups, \textit{Riv. Mat. Univ. Parma } \textbf{8} (2017), 353-366.

%%%%%%%%%

\bigskip


\begin{thebibliography}{99}
\bibitem{DDM} U. Dardano and F. De Mari, A real chain condition for groups, submitted (\href{https://doi.org/10.48550/arXiv.2307.07639}{\scriptsize  \rm https://doi.org/10.48550/arXiv.2307.07639})

\bibitem{D} M.R. Dixon, \textit{Sylow theory, formations and Fitting classes
in locally finite groups}, World Scientific, Singapore (1994).

\bibitem{DES} M.R. Dixon, M.J. Evans and H. Smith, Locally soluble-by-finite
groups with the weak minimal condition on non-nilpotent subgroups.\textit{\
J. Algebra} \textbf{249 }(2002), 226-246.
%
\bibitem{DS} M.R. Dixon and I.Y. Subbotin, Groups with finiteness conditions
on some subgroup system: a contemporary stage, \textit{Algebra Discrete
Math. }\textbf{4} (2009), 29-54.
%%
%\bibitem{F73} L. Fuchs, \textit{Abelian groups}, Springer, Berlin-New York
%(2015)
%
\bibitem{GKR} F. de Giovanni, L.A. Kurdachenko and A. Russo, Groups with
pronormal deviation, \textit{J. Algebra} \textbf{613}(2023), 32-45.

\bibitem{dGV} F. de Giovanni and G. Vincenzi, Some topics in the theory of
pronormal subgroups of groups, in: \textit{Topics in infinite groups}, 
\textit{Quad. Mat. }\textbf{8 }(Dept. Math., Seconda Univ. Napoli, Caserta,
2001), pp. 175-202.

\bibitem{dGV95} F. de Giovanni and G. Vincenzi, Groups satisfying the
minimal condition on non-pronormal subgroup, \textit{Boll. Un. Mat. Ital. A}
(7) \textbf{9} (1995), 185-194.

\bibitem{dGV00} F. de Giovanni and G. Vincenzi, Pronormality in infinite
groups, \textit{Math. Proc. R. Ir. Acad.} \textbf{100A} (2000), 189-203.

\bibitem{KS} N. F. Kuzenny\u{\i} and I. Ya. Subbotin, Groups in which all
subgroups are pronormal, \textit{Ukrainian Math. J.} \textbf{39} (1987),
251-254.

\bibitem{CR} J.C. McConnell and J.C. Robson, \textit{Noncommutative
Noetherian Rings}, John Eiley and Sons, Chichester (1988).

\bibitem{P} T.A. Peng, Finite groups with pronormal subgroups, \textit{Proc.
Amer. Math. Soc.} \textbf{20} (1969), 232-234.

\bibitem{R64} D.J.S. Robinson, Groups in which normality is a transitive
relation, \textit{Proc. Cambridge Philos. Soc.} \textbf{60} (1964), 21-38.

\bibitem{R72} D.J.S. Robinson, \textit{Finiteness conditions and gerenalized
soluble groups}, Springer, Berlin-New York (1972).

\bibitem{R96} D.J.S. Robinson, \textit{A course in the theory of groups},
Springer, Berlin-New York (1996).

%\bibitem{S} T. Szele, On the basic subgroups of abelian $p$-groups. \textit{Acta Math. Acad. Sci. Hungar} \textbf{5} (1954), 129-141.


\bibitem{T03} A.V. Tushev, On deviation in groups. \textit{\ Illinois J.
Math. } \textbf{47} (2003),539-550.


\bibitem{V1} G. Vincenzi, Groups with dense pronormal subgroups, \textit{%
Ricerche Mat.} \textbf{40} (1991), 75-79.

\bibitem{V2} G. Vincenzi, Groups satisfying the maximal condition on
non-pronormal subgroups, \textit{Algebra Colloq.} \textbf{5} (1998), 121-134.


\bibitem{Z} D.I. Zaicev, On the theory of minimax groups, \textit{Ukrainian
Math. J. }\textbf{23} (1971), 536-542.
\end{thebibliography}
\end{document}